\documentclass[a4paper,10pt]{article}

\usepackage{graphicx}

\usepackage{amssymb}
\usepackage{amsmath}
\usepackage[english]{babel}
\usepackage{color}

\newenvironment{proof}{\paragraph{Proof:}}{\hfill$\square$}

\newtheorem{lemma}{Lemma}[section]
\newtheorem{corollary}[lemma]{Corollary}
\newtheorem{theorem}[lemma]{Theorem}
\newtheorem{remark}[lemma]{Remark}

\newtheorem{proposition}[lemma]{Proposition}

\newcommand{\R}{\mathbb{R}}

\newcommand{\N}{\mathbb{N}}
\newcommand{\eps}{\varepsilon}

\newtheorem{hyp}{Hypotheses}

\begin{document}

\title{Kernel estimates for\\ Schr\"odinger type operators\\ with unbounded diffusion and potential terms}

\author{Anna Canale, Abdelaziz Rhandi and Cristian Tacelli}


\maketitle

\abstract{
We prove that the heat kernel associated to the Schr\"odinger type operator
$A:=(1+|x|^\alpha)\Delta-|x|^\beta$ satisfies the estimate
$$k(t,x,y)\leq c_1e^{\lambda_0t}e^{c_2t^{-b}}\frac{(|x||y|)^{-\frac{N-1}{2}-\frac{\beta-\alpha}{4}}}{1+|y|^\alpha}
e^{-\frac{2}{\beta-\alpha+2}|x|^{\frac{\beta-\alpha+2}{2}}}
e^{-\frac{2}{\beta-\alpha+2}|y|^{\frac{\beta-\alpha+2}{2}}}
$$
for $t>0,|x|,|y|\ge 1$, where $c_1,c_2$ are positive constants and $b=\frac{\beta-\alpha+2}{\beta+\alpha-2}$ provided that
$N>2,\,\alpha\geq 2$ and $\beta>\alpha-2$. We also obtain an estimate of the eigenfunctions of $A$.
}
\\{\bf keywords.}{ Schr\"odinger type operator,  Semigroup, heat kernel estimates}

\section{Introduction}

In this paper we consider the operator
\[
Au(x)=(1+|x|^\alpha)\Delta u(x) -|x|^\beta u(x),\quad x\in \R^N,
\]
for $N>2,\,\alpha\geq 2$ and $\beta>\alpha-2$. We propose to study the behaviour of the associated heat kernel and associated eigenfunctions.

Recently several paper have dealt with elliptic operators with polynomially
growing diffusion coefficients
(see for example \cite{met-spi2}, \cite{met-spi3}, \cite{met-spi-tac}, \cite{met-spi-tac2},
\cite{lor-rhan}, \cite{for-lor},
\cite{can-rhan-tac1}, \cite{kunze-luc-rha1}, \cite{kunze-luc-rha2},
\cite{can-tac1},
\cite{dur-man-tac1}).

In \cite{lor-rhan} (resp. \cite{can-rhan-tac1})
it is proved that the realization $A_p$
of $A$ in $L^p(\R^N)$ for $1<p<\infty$ with domain
$$
D_p(A)=\{u\in W^{2,p}(\R^N)\;|\; (1+|x|^\alpha)|D^2u|,(1+|x|^\alpha)^{1/2}\nabla u,|x|^\beta u\in L^p(\R^N) \}
$$
 generates a strongly continuous and analytic semigroup $T_p(\cdot)$ for $\alpha\in [0,2]$ and $\beta >0$ (resp. $\alpha>2$ and $\beta >\alpha -2$).
This semigroup is also consistent, irreducible and ultracontractive. For the case $\beta=0$ we refer to \cite{for-lor} and \cite{met-spi2}.

Since the coefficients of the operator $A$ are locally regular it follows that the semigroup $T_p(\cdot)$
admits an integral representation through a heat kernel $k(t,x,y)$
\[
 T_p(t)u(x)=\int_{\R^N} k(t,x,y)u(y)dy\;,\qquad t>0,\;x\in \R^N\;,
\]
for all $u\in L^p(\R^N)$ (cf. \cite{ber-lor}, \cite{met-wack}).

In \cite{lor-rhan} estimates of the kernel $k(t,x,y)$ for $\alpha\in [0,2)$ and $\beta>2$ were obtained.
Our contribution in this paper is to show similar upper bounds for the case $\alpha\geq 2$ and $\beta>\alpha-2$.
Our techniques consist in providing upper and lower estimates for the ground state of $A_p$ corresponding to the largest eigenvalue $\lambda_0$
and adapting the arguments used in \cite{davies}.
\medskip

The paper is structured as follows. In section \ref{sc:eigenfunction} we prove that the eigenfunction $\psi(x)$ associated to 
the largest eigenvalue $\lambda_0$ can be estimated from below and above by the function 
$$|x|^{-\frac{N-1}{2}-\frac{\beta-\alpha}{4}} e^{-\frac{2}{\beta-\alpha+2}|x|^{\frac{\beta-\alpha+2}{2}}}$$
for $|x|\ge 1$.

In Section \ref{sc:estimates-eigenfunction} we  introduce the
measure $d\mu=(1+|x|^\alpha)^{-1}dx$ for which the operator $A$ is symmetric
and generates an analytic semigroup (which is a Markov semigroup) with kernel
\[
k_\mu(t,x,y)=(1+|x|^\alpha)k(t,x,y).
\]
Adapting the arguments used in \cite{davies} and \cite{lor-rhan}, we show
the following intrinsic ultracontractivity
\[
k_\mu(t,x,y)\leq c_1 e^{\lambda_0 t}e^{c_2 t^{-b}}\psi(x)\psi(y), \quad t>0,\,x,\,y\in \R^N,
\]
where $c_1,c_2$ are positive constant, $b=\frac{\beta-\alpha+2}{\beta+\alpha-2}$, provided that $N>2,\,\alpha\geq 2$ and $\beta>\alpha-2$. So one deduces the heat kernel estimate
\begin{align*}
&k(t,x,y)\leq c_1e^{\lambda_0t}e^{c_2t^{-b}}\\
&\qquad \times 
\left(|x|^{-\frac{N-1}{2}-\frac{\beta-\alpha}{4}} 
    e^{-\frac{2}{\beta-\alpha+2}|x|^{\frac{\beta-\alpha+2}{2}}}
    \right)
\left(\frac{|y|^{-\frac{N-1}{2}-\frac{\beta-\alpha}{4}}}{1+|y|^\alpha}
    e^{-\frac{2}{\beta-\alpha+2}|y|^{\frac{\beta-\alpha+2}{2}}}
    \right)
\end{align*}
for $t>0,|x|,|y|\ge 1$. As an application we obtain the behaviour of all eigenfunctions of $A_p$ at infinity. With respect to $t$ we prove the following sharp estimates
\begin{align*}
k_\mu(t,x,y)\leq Ct^{-\frac{N}{2}}\left(1+|x|^\alpha \right)^{\frac{2-N}{4}}\left(1+|y|^\alpha \right)^{\frac{2-N}{4}}
\end{align*}
for $0<t\le 1$ and $x,\,y\in \R^N$. Here we use the results in \cite{met-spi3} and weighted Nash inequalities introduced in \cite{bakry}.
We end this section by giving a brief description of how to extend the heat kernel estimates to a more general class of elliptic operators in divergence form.

In the sequel we denote by $B_R\subset\R^N$ the open ball, centered at $0$ with radius $R>0$.

\section{Estimate of the ground state $\psi$}\label{sc:eigenfunction}
We begin by estimating the eigenfunction $\psi$ corresponding to the largest eigenvalue $\lambda_0$ of $A$.
First we recall some spectral properties obtained in \cite{can-rhan-tac1} and \cite{lor-rhan}.
\begin{proposition} \label{comp}
Assume $N>2,\,\alpha\geq2$  and $\beta>\alpha-2$ then
\begin{enumerate}
\item[(i)] the resolvent of $A_p$ is compact in $L^p(\R^N)$;
\item[(ii)] the spectrum of $A_p$ consists of a sequence of negative real eigenvalues which accumulates at $-\infty$.
 Moreover, $\sigma(A_p)$ is independent of $p$;
\item[(iii)] the semigroup $T_p(\cdot)$ is irreducible, the eigenspace corresponding to the largest eigenvalue
    $\lambda_0$ of $A_p$ is one-dimensional and is spanned by strictly positive functions $\psi$,
 which is radial, belongs to $C_b^{1+\nu}(\R^N)\cap C^2(\R^N)$ for any $\nu \in (0,1)$ and tends to $0$ when $|x|\to
 \infty$.
\end{enumerate}
\end{proposition}

We can now prove upper and lower estimates for $\psi$. We note here that the proof of \cite[Proposition 3.1]{lor-rhan} cannot be adapted to our situation. So, we propose to use another technique to estimate $\psi$.

\begin{proposition}   \label{primautofunz2}
Let $\lambda_0<0$ be the largest eigenvalue of $A_p$ and $\psi$ be the
corresponding eigenfunction. If $N>2,\,\alpha \geq 2$ and $\beta>\alpha-2$ then
$$
C_1|x|^{-\frac{N-1}{2}-\frac{\beta-\alpha}{4}}
    e^{-\frac{2}{\beta-\alpha+2}|x|^{\frac{\beta-\alpha+2}{2}}}
\le \psi(x)\leq 
C_2 |x|^{-\frac{N-1}{2}-\frac{\beta-\alpha}{4}}
    e^{-\frac{2}{\beta-\alpha+2}|x|^{\frac{\beta-\alpha+2}{2}}}
$$
for any $x\in\R^N\setminus B_1$ and some positive constants $C_1,\,C_2$.
\end{proposition}
\begin{proof}
Since the eigenfunction is radial, we have to study the asymptotic behavior of the solution of an ordinary differential equation.
We follow the idea of the WKB method (see \cite{olver}), but since the error function is not bounded we need to compute it directly.

Let $f_{\alpha,\beta,\lambda}$ be the function
\begin{equation}\label{f}
f_{\alpha,\beta,\lambda}(x)=|x|^{-\frac{N-1}{2}}h^{-\frac{1}{4}}(|x|)
\exp \left\{-\int_R^{|x|}h^{\frac{1}{2}}(s)ds -\int_R^{|x|}v_\lambda(s)ds \right\}\;,
\end{equation}
where
$\lambda \in \R ,\,h(r)=\frac{r^\beta}{1+r^\alpha}$,
and $v_\lambda$ is a smooth function to be chosen later on.
If we set
\begin{equation}\label{w}
w(r)=r^{\frac{N-1}{2}}f_{\alpha,\beta,\lambda}(r),
\end{equation}
then
\begin{equation}\label{w''1}
 w'=w\left( -\frac{h'}{4h}-h^{\frac{1}{2}}-v_\lambda\right)
\qquad
\hbox{and}\qquad
w''= w(g+m+h),
\end{equation}
where
\begin{equation}\label{R}
g=\frac{5}{16}\left(\frac{h'}{h}\right)^2-
\frac{h''}{4h}+v_\lambda^2+v_\lambda\left(\frac{h'}{2h}+
2h^{\frac{1}{2}}\right)-v_\lambda '-m
\end{equation}
and
$$
m(r):=\frac{(N-1)(N-3)}{4r^2}\,.
$$
On the other hand, taking in mind (\ref{w}) we also obtain
\begin{equation}\label{w''2}
w''(r)=r^{\frac{N-1}{2}}\left( f_{\alpha,\beta,\lambda}'' +
\frac{N-1}{r}f_{\alpha,\beta,\lambda}'+\frac{(N-1)(N-3)}{4r^2}f_{\alpha,\beta,\lambda} \right).
\end{equation}
Comparing (\ref{w''1}) and (\ref{w''2}) we get
\[
f_{\alpha,\beta,\lambda}'' +\frac{N-1}{r}f_{\alpha,\beta,\lambda}'=
\frac{r^\beta}{1+r^\alpha}f_{\alpha,\beta,\lambda}+gf_{\alpha,\beta,\lambda}.
\]
That is
\begin{equation}\label {Delta}
\Delta f_{\alpha,\beta,\lambda}(x)-\frac{|x|^\beta}{1+|x|^\alpha} f_{\alpha,\beta,\lambda}(x)
 =g(|x|)f_{\alpha,\beta,\lambda}(x).
\end{equation}
To evaluate the function $g$ we set $\xi=\frac{\beta-\alpha}{2}+1$,
which is positive by the condition $\beta>\alpha-2$.
We have
\[
\frac{h'}{h}=\frac{1}{r}(\beta-\alpha)+\frac{1}{r}O(r^{-\alpha}),
\,\,\frac{h''}{h}=\frac{1}{r^2}(\beta-\alpha)(\beta-\alpha-1)+\frac{1}{r^2}O(r^{-\alpha}).
\]
Then (\ref{R}) is reduced to
\begin{eqnarray}\label{eq:errore-R2}
g(r) &=& -v_\lambda '+\frac{v_\lambda}{r}
    \left( \xi-1+O(r^{-\alpha})+2r^\xi \sqrt{\frac{r^\alpha}{1+r^\alpha}}\right)+v_\lambda^2 \nonumber \\
    & & +\frac{c_0}{r^2}+\frac{1}{r^2}\left(O(r^{-\alpha})+O(r^{-2\alpha})\right)\nonumber \\
&=& -v_\lambda '
+\frac{v_\lambda}{r}\left(
	\xi-1+O(r^{-\alpha})+2r^\xi -2r^{\xi} \frac{  (1+r^\alpha)^{1/2}-r^{ \alpha/2 }   }{  (1+r^\alpha)^{1/2} }
	\right)\nonumber \\
&&+v_\lambda^2 +\frac{c_0}{r^2}+\frac{1}{r^2}O(r^{-\alpha})\nonumber \\
&=& -v_\lambda '
+\frac{v_\lambda}{r}\left( \xi-1+2r^\xi +(1+r^\xi)O(r^{-\alpha})\right)
+v_\lambda^2\nonumber \\
& & +\frac{c_0}{r^2}+\frac{1}{r^2}O(r^{-\alpha}),
\end{eqnarray}
where $c_0=c_0(\xi)=\left( \frac{\xi-1}{2} \right)^2+\frac{\xi-1}{2}
-\frac{(N-1)(N-3)}{4}$. 
So, if we take in (\ref{eq:errore-R2})
\[
 v_\lambda(r)=\sum_ {i=1}^k c_i\frac{1}{r^{i\xi+1}},\quad r\ge 1,
\]
we obtain
\begin{eqnarray*}
r^2g(r)&=& \sum_{i=1}^k c_i(i\xi+1)\frac{1}{r^{i\xi}}+
(\xi-1)\sum_{i=1}^kc_i\frac{1}{r^{i\xi}}+2  \sum_{i=0}^{k-1} c_{i+1}\frac{1}{r^{i\xi}}\\
 & & + \left(\sum_{i=1}^kc_i\frac{1}{r^{i\xi}}+ \sum_{i=0}^{k-1} c_{i+1}\frac{1}{r^{i\xi}}\right)O(r^{-\alpha})\\
 & & +\sum_{i,j=1}^kc_ic_j\frac{1}{r^{(i+j)\xi}}
  +c_0+O(r^{-\alpha})\\
&=&\sum _{i=2}^{k-1}\left[c_i\xi(i+1)+2c_{i+1}+\sum_{j+s=i}c_jc_s \right]\frac{1}{r^{i\xi}}+(2c_1\xi+2c_2)r^{-\xi}\\
& &  +c_k\xi (k+1)\frac{1}{r^{k\xi}}+2c_1+\sum _{i+j\geq k}\frac{c_ic_j}{r^{(i+j)\xi}}
  +c_0+O(r^{-\alpha}).
\end{eqnarray*}
We can choose $c_1,\dots,c_{k}$ such that $$2c_1+c_0=\lambda ,\,2c_1\xi +2c_2=0
\hbox{\ and }\left[\xi(i+1)c_i+2c_{i+1}+\sum_{j+s=i}c_jc_s \right]=0$$ for $i=2,\cdots, k-1$
and obtain
\begin{align}
r^2g(r)=
\lambda+c_k\xi (k+1)\frac{1}{r^{k\xi}}+\sum_{i+j\ge k}\frac{c_ic_j}{r^{(i+j)\xi}}+ O(r^{-\alpha}).
\end{align}
Thus,
\[
g(r)=O\left(\frac{1}{r^{k\xi+2} }\right)+O\left(\frac{1}{r^{\alpha +2} }\right)+\frac{\lambda}{r^2}.
\]
Since $\xi>0$ there exists $k\in \N$ such that $k\xi+2-\alpha>0$. So we have

\begin{equation}\label{eq:stima-f-alpha-beta-lambda0}
(1+|x|^{\alpha}) \Delta f_{\alpha,\beta,\lambda}(x)-|x|^{\beta} f_{\alpha,\beta,\lambda}(x)
  =o(1)f_{\alpha,\beta,\lambda}(x)+\lambda \frac{1+|x|^{\alpha}}{|x|^2}f_{\alpha,\beta,\lambda}(x).
\end{equation}
We prove first the upper bound.

For $\psi$ we know that
\begin{equation}\label{eq:eq-aut-diviso}
 \Delta \psi-\frac{|x|^\beta}{1+|x|^\alpha}\psi-\frac{\lambda_0}{1+|x|^\alpha}\psi =0\;.
\end{equation}
Since $\alpha-2\geq 0$ and $\lambda_0<0$, for $|x|$ large enough we have $o(1)+2\lambda_0 \frac{1+|x|^{\alpha}}{|x|^2}<\lambda_0$.
Then, by \eqref{eq:stima-f-alpha-beta-lambda0}, it follows that
\[
(1+|x|^{\alpha}) \Delta f_{\alpha,\beta,2\lambda_0}(x)-|x|^{\beta} f_{\alpha,\beta,2\lambda_0}(x)
  < \lambda_0 f_{\alpha,\beta,2\lambda_0}.
\]
Thus,
\begin{equation}\label{eq:f-2lambda-diviso}
\Delta f_{\alpha,\beta,2\lambda_0}(x)-\frac{|x|^{\beta}}{1+|x|^\alpha} f_{\alpha,\beta,2\lambda_0}(x)
   -\frac{\lambda_0}{1+|x|^\alpha} f_{\alpha,\beta,2\lambda_0}(x)< 0\;,
\end{equation}
in $\R^N\setminus B_R$ for some $R>0$.
Comparing \eqref{eq:eq-aut-diviso} and \eqref{eq:f-2lambda-diviso}, in $\R^N\setminus B_R$
we have
\begin{eqnarray*}
\Delta (f_{\alpha,\beta,2\lambda_0}-C\psi)
& < &
\frac{\lambda_0+|x|^\beta}{1+|x|^{\alpha}}(f_{\alpha,\beta,2\lambda_0}-C\psi)\quad \hbox{\ for any constant }C>0.
\end{eqnarray*}
Since $\beta>0$, we have
$$
W(x):=\frac{\lambda_0+|x|^\beta}{1+|x|^{\alpha}}>0
$$
for $|x|$ large enough.
Since both $f_{\alpha,\beta,2\lambda_0}$ and $\psi$ tend to $0$ as $|x|\to \infty$
and since
there exists $C_2$ such that $\psi\le C_2 f_{\alpha,\beta,2\lambda_0}$ on $\partial B_R$, we can apply
the maximum principle to the problem
$$\left\{\begin{array}{ll}
\Delta g(x)-W(x)g(x)<0 \,\,\quad \hbox{\ in }\R^N\setminus B_R ,\\
g(x)\ge 0 \hspace*{3cm}\hbox{\ in }\partial B_R ,\\
\lim_{|x|\to \infty}g(x)=0,
\end{array} \right.
$$ where $g:=f_{\alpha,\beta,2\lambda_0}-C_2^{-1}\psi$,
to obtain
$\psi\leq C_2 f_{\alpha,\beta,2\lambda_0}$ in $\R^N\setminus B_R$. Here one has to note that since $\lim_{|x|\to \infty}g(x)=0$, one can see that the classical maximum principle on bounded domains can be applied, cf. \cite[Theorem 3.5]{gil-tru}.
Then,
\begin{align*}
&\psi(x)\leq C_2 |x|^{-\frac{N-1}{2}-\frac{1}{4}(\beta-\alpha) }
\exp\left\{-\int_R^{|x|}\sqrt{\frac{r^\beta}{1+r^\alpha}}\,dr\right\}
\exp\left\{-\int_R^{|x|}v_{2\lambda_0}(r)dr\right\}\\
&\quad\leq C_3
|x|^{-\frac{N-1}{2}-\frac{1}{4}(\beta-\alpha) }
\exp\left\{-\int_R^{|x|}\sqrt{\frac{r^\beta}{1+r^\alpha}}\,dr\right\},
\end{align*}
since
\begin{eqnarray}\label{extra-justi}
\lim_{|x|\to \infty}\int_R^{|x|} v_{\lambda}(r)\,dr &=&
\lim_{|x|\to \infty}\sum_{j=1}^k\frac{c_j}{j\xi}(R^{-j\xi}-|x|^{-j\xi})\nonumber \\
&=& \sum_{j=1}^k\frac{c_j}{j\xi}R^{-j\xi}.
\end{eqnarray}

As regards lower bounds of $\psi $, we observe that, from
\eqref{eq:stima-f-alpha-beta-lambda0}, we have
\[
\Delta f_{\alpha,\beta,0}(x)-\frac{|x|^{\beta}}{1+|x|^\alpha}f_{\alpha,\beta,0}(x)=
\frac{o(1)}{1+|x|^\alpha}f_{\alpha,\beta,0}(x)>
\frac{\lambda_0}{1+|x|^\alpha}f_{\alpha,\beta,0}(x)
\]
if $|x|\ge R$ for some suitable $R>0$.
Then,
\[
\Delta f_{\alpha,\beta,0}(x)>
\frac{|x|^{\beta}}{1+|x|^\alpha}f_{\alpha,\beta,0}(x)+
\frac{\lambda_0}{1+|x|^\alpha}f_{\alpha,\beta,0}(x)
\]
Since $\frac{\lambda_0}{1+|x|^\alpha}\psi=
\Delta \psi (x)-\frac{|x|^\beta}{1+|x|^\alpha}\psi$ we have
\begin{align*}
&\Delta (f_{\alpha,\beta,0}-\psi)
> \frac{|x|^\beta+\lambda_0}{1+|x|^\alpha}(f_{\alpha,\beta,0}-\psi)\;.
\end{align*}
We can assume that $|x|^\beta+\lambda_0$ is positive for $|x|\ge R$ and, arguing as above,
by the maximum principle and using \eqref{extra-justi} we get
\[
\psi(x) \ge C_1f_{\alpha,\beta,0}(x)\geq C_1 |x|^{-\frac{N-1}{2}-\frac{1}{4}(\beta-\alpha) }
\exp\left\{-\int_R^{|x|} \sqrt{\frac{r^\beta}{1+r^\alpha}}\, dr\right\}
\]
for $|x|\ge R$. Since $0<\psi \in C(\R^N)$, by changing the constants, the above upper and lower estimates remain valid for $1\le |x|\le R$. This ends the proof of the proposition.
\end{proof}

\section{Intrinsic ultracontractivity and heat kernel estimates}\label{sc:estimates-eigenfunction}
Let us now introduce on $L^2_\mu:=L^2(\R^N,d\mu)$ the bilinear form
\begin{equation}\label{formaquad}
a_\mu(u,v)=\int_{\R^N}\nabla u\cdot \nabla \overline{v}\,dx+\int_{\R^N}Vu\overline{v}\,d\mu,\qquad\;\, u,v\in D(a_\mu),
\end{equation}
where $V(x)=|x|^\beta ,\,d\mu(x)=(1+|x|^\alpha)^{-1}dx$ and $D(a_\mu)=\overline{C_c^\infty(\R^N)}^{\|\cdot \|_H}$ with $H$ the Hilbert space $$H=
\{u\in L^2_\mu: V^{1/2}u\in L^2_{\mu},\, \nabla u\in (L^2(\R^N))^N\}$$ endowed with the inner product
$$
\langle u,v\rangle_{H}=\int_{\R^N}(1+V)u\overline{v}\,d\mu +\int_{\R^N}\nabla u\cdot \nabla \overline{v}\,dx.
$$
Since $a_{\mu}$ is a closed, symmetric and accretive form, to $a_\mu$ we associate the self-adjoint operator $A_\mu$ defined by
\begin{eqnarray*}
\begin{array}{l}
D(A_\mu)=\displaystyle\left\{u\in D(a_\mu) : \exists g\in L^2_\mu\,\, \mbox{ s.t. } a_\mu(u,v)=-\int_{\R^N}g\overline{v}\,d\mu,\,\forall v\in D(a_\mu)\right\},\\[3mm]
A_\mu u=g,
\end{array}
\end{eqnarray*}
see e.g., \cite[Prop. 1.24]{ouhabaz}.
By general results on positive self-adjoint operators induced by nonnegative quadratic forms
in Hilbert spaces (see e.g., \cite[Prop. 1.51, Thms. 1.52, 2.6, 2.13]{ouhabaz})
$A_\mu$ generates a positive analytic semigroup $(e^{tA_\mu})_{t\ge 0}$ in $L^2_\mu$.

We need to show that the semigroup $e^{t{A_\mu}}$ coincides with the semigroup $T_p(\cdot)$ generated
by $A_p$ in $L^p(\R^N)$ on $L^p(\R^N) \cap L^2_\mu$.

\begin{lemma} \label{coerenza1} We have
$$
D(A_\mu)= \left\{u \in D(a_\mu) \cap W^{2,2}_{loc}(\R^N):
(1+|x|^\alpha) \Delta u -V(x)u \in L^2_\mu \right\}
$$
and $A_\mu u=(1+|x|^\alpha)\Delta u-V(x)u $ for $u \in D(A_\mu)$. Moreover, if
$\lambda >0$ and $f \in L^p(\R^N) \cap L^2_\mu$, then
$$
(\lambda-A_\mu)^{-1}f=(\lambda-A_p)^{-1}f.
$$
\end{lemma}
\begin{proof}
The inclusion $"\subset "$ is obtained, taking $v\in C_c^{\infty}(\R^N)$ in \eqref{formaquad}, by local elliptic regularity.
As regards the inclusion $"\supset"$ we consider $u\in D(a_\mu) \cap W^{2,2}_{loc}(\R^N)$ such that $g:=(1+|x|^\alpha)\Delta u-V(x)u\in L^2_\mu$ and consider
$v\in C_c^{\infty}(\R^N)$. Integrating by parts we obtain
\begin{equation}\label{eq:int-part}
a_\mu(u,v)=-\int gvd\mu.
\end{equation}
By the density of $C_c^{\infty}(\R^N)$ in $D(a_\mu)$ we have equation \eqref{eq:int-part} for every $v\in D(a_\mu)$. This implies that $u \in D(A_\mu).$

To show the coherence of the resolvent, we consider $f\in C_c^{\infty}(\R^N)$
and let $u=(\lambda-A)^{-1}f$. Since $f\in L^2(\R^N)\cap C_0(\R^N)$, by
 \cite[Theorem 3.7]{can-rhan-tac1} and \cite[Theorem 4.4]{can-rhan-tac1}, it follows that $u\in D_{2}(A)$. So, we have $\nabla u\in L^2(\R^N)$
and $Vu\in L^2(\R^N)$. Moreover, it is clear that $u\in L^2_\mu$, and
\begin{eqnarray}\label{eq:pot-1-2}
\|V^{1/2}u\|_{L^2_\mu}^2 &\leq &  \int_{\R^N}V(x)u^2dx \nonumber \\
&\leq &
\int_{B(1)}u^2 dx+\int_{\R^N/B(1)}V^2(x)u^2dx \nonumber \\
&\le & \|u\|_2^2+\|Vu\|_2^2.
\end{eqnarray}
This yields $u\in H$.
Since $C_c^{\infty}(\R^N)$ is dense in $D_{2}(A)$, see \cite[Lemma 4.3]{can-rhan-tac1}, we can find a sequence $(u_n)\subset C_c^\infty(\R^N)$
such that $u_n$ converges to $u$ in the operator norm.
Then, $u_n$ converges to $u$ in $L^2(\R^N)$ and hence in $L_\mu^2$.
By \cite[Lemma 4.2]{can-rhan-tac1} $\nabla u_n$ converges to $\nabla u_n$ in $L^2(\R^N)$ and hence in $L_\mu^2$.
Finally replacing $u$ with $u_n-u$ in \eqref{eq:pot-1-2} we have that $V^{1/2}u_n$ converges to $V^{1/2}u$ in $L^2_\mu$.
Thus we have proved that $u\in D(a_\mu)$.
Integration by parts we obtain
\[
 a(u,v)=-(\lambda u-f,v)_{L^2_\mu}.
\]
That is $u\in D(A_\mu)$ and $\lambda u-A_\mu u=f$. Therefore, $(\lambda -A_\mu)^{-1}f=(\lambda -A_p)^{-1}f$ for all $f\in C_c^\infty(\R^N)$ and so by density the last statement follows.
\end{proof}

The previous Lemma implies in particular that
\[
e^{tA_\mu}f=T_p(t)f=\int_{\R^N}k(t,x,y)f(y)\,dy,\qquad f\in L^p(\R^N)\cap L_\mu^2.
\]
By density we obtain that the semigroup $e^{tA_\mu}$ admits
the integral representation  $e^{tA_\mu}f(x)=\int_{\R^N}k_\mu(t,x,y)f(y)d\mu(y)$ for all $f\in L^2_\mu$, where
\begin{equation}\label{eq:relaz-semigruppi}
k_\mu(t,x,y)=(1+|y|^\alpha)k(t,x,y),\qquad t>0,\, x,y\in \R^N.
\end{equation}

Let us now give the first application of Proposition \ref{primautofunz2}. The proof is similar to the one given in \cite[Proposition 3.4]{lor-rhan} and is based on the semigroup law and the symmetry of $k_\mu(t,\cdot ,\cdot)$ for $t>0$.
\begin{proposition}\label{on-diagonal}
If $N>2,\,\alpha \ge 2$ and $\beta >\alpha -2$, then
\begin{eqnarray*}
k(t,x,x)\ge Me^{\lambda_0 t}\left(|x|^{\frac{\alpha-\beta}{4}-\frac{N-1}{2}}
e^{-\frac{2}{\beta-\alpha+2}|x|^{\frac{\beta-\alpha+2}{2}}}\right)^2(1+|x|^\alpha)^{-1},\quad t>0,
\end{eqnarray*}
for all $x\in\R^N\setminus B_1$ and some constant $M>0$.
\end{proposition}

We now give the main result of this section.
\begin{theorem}\label{th:stima-kernel-via-autofunz}
If $N>2$, $\alpha\geq 2$ and $\beta>\alpha-2$ then
$$
k(t,x,y)\leq c_1e^{\lambda_0t+c_2t^{-b}}|x|^{-\frac{N-1}{2}-\frac{\beta-\alpha}{4}} 
  e^{-\frac{2}{\beta-\alpha+2}|x|^{\frac{\beta-\alpha+2}{2}}}
\frac{|y|^{-\frac{N-1}{2}-\frac{\beta-\alpha}{4}}}{1+|y|^\alpha}
  e^{-\frac{2}{\beta-\alpha+2}|y|^{\frac{\beta-\alpha+2}{2}}}
$$
for $t>0,\,x,y\in \R^N\setminus B_1$, where $c_1,c_2$ are positive constants and $b=\frac{\beta-\alpha+2}{\beta+\alpha-2}$.
\end{theorem}
\begin{proof}
Let us prove first
\begin{equation}\label{estim-(0,1]}
k(t,x,y)\leq c_1e^{c_2t^{-b}}|x|^{-\frac{N-1}{2}-\frac{\beta-\alpha}{4}}
e^{-\frac{2}{\beta-\alpha+2}|x|^{\frac{\beta-\alpha+2}{2}}}
\frac{|y|^{-\frac{N-1}{2}-\frac{\beta-\alpha}{4}}}{1+|y|^\alpha}
e^{-\frac{2}{\beta-\alpha+2}|y|^{\frac{\beta-\alpha+2}{2}}}
\end{equation}
for $0<t \le 1,\,x,y\in \R^N\setminus B_1$.
By adapting the arguments used in \cite[Subsect. 4.4 and 4.5]{davies}, we have only to show the following estimates
\begin{equation}\label{ultra-mu}
\int_{\R^N}g|u|^2d\mu
	\leq C \|g\|_{L_\mu^{N/2}}a_{\mu}(u,u),\quad u\in D(a_\mu),\ g\in L^{N/2}_\mu ,
\end{equation}
and
\begin{equation}\label{ip-2-rosen}
 \int _{\R^N} -\log \psi |u|^2d\mu\leq \eps a_\mu(u,u)+(C_1\eps^{-b}+C_2)\|u\|^2_{L^2_\mu},\quad u\in D(a_\mu).
\end{equation}
To prove \eqref{ultra-mu} we observe that using  H\"older and Sobolev inequality we obtain
\begin{eqnarray*}
\int_{\R^N}g|u|^2d\mu &\leq & C\left(\int_{\R^N}|g|^{N/2} d\mu\right)^{2/N}
\left(\int_{\R^N}|u|^{\frac{2N}{N-2}} d\mu \right)^{\frac{N-2}{N}}\\
&=& C\| g\|_{L^{N/2}_\mu}\|u\|^2_{L^{2^*}_\mu}\\
&\leq & C\| g\|_{L^{N/2}_\mu}\|\nabla u\|^2_{2}
				\leq C\| g\|_{L^{N/2}_\mu} a_{\mu}(u,u), \quad u\in D(a_\mu).
\end{eqnarray*}

To show \eqref{ip-2-rosen}, we apply the lower estimate of $\psi$ obtained in Proposition \ref{primautofunz2}
\begin{eqnarray*}
 -\log\psi &\leq &-\left(\log C_1 -\frac{2}{\beta -\alpha +2}\right)+\frac{2N-2+\beta -\alpha}{4}\log |x| \\
 & & \quad +\frac{2}{\beta -\alpha +2}|x|^{\frac{\beta-\alpha+2}{2}}
\end{eqnarray*}
for $|x|\ge 1$.
Hence, there are positive constants $C_1,\,C_2$ such that
$$-\log\psi \leq C_1|x|^{\frac{\beta-\alpha+2}{2}}+C_2,\quad x\in \R^N.$$
Since $\xi=\frac{\beta-\alpha}{2}+1<\beta$ we have
\[
|x|^\xi\leq \eps |x|^\beta+C\eps^{-\frac{\xi}{\beta-\xi}}=\eps V(x)+C\eps^{-b}
\]
for all $\eps >0$.
Thus,
\[
 -\log\psi\leq \eps V+c_1\eps^{-b}+c_2.
\]
Taking into account that
$\int_{\R^N} V|u|^2d\mu\leq a_\mu(u,u)$ for all $u\in D(a_\mu)$, we obtain \eqref{ip-2-rosen}. This end the proof of \eqref{estim-(0,1]}.

It remains to prove that
\begin{equation}
k(t,x,y)\le Ce^{\lambda_0t}|x|^{-\frac{N-1}{2}-\frac{\beta-\alpha}{4}} 
e^{-\frac{2}{\beta-\alpha+2}|x|^{\frac{\beta-\alpha+2}{2}}}
\frac{|y|^{-\frac{N-1}{2}-\frac{\beta-\alpha}{4}}}{1+|y|^\alpha}
e^{-\frac{2}{\beta-\alpha+2}|y|^{\frac{\beta-\alpha+2}{2}}}
\label{tge1}
\end{equation}
for $t>1,\,x,y\in\R^N\setminus B_1$ and some constant $C>0$.
To this purpose we use the semigroup law and the symmetry of $k_\mu(t,\cdot ,\cdot)$ to infer that
\begin{eqnarray*}
k_{\mu}(t,x,y)=\int_{\R^N}k_{\mu}(t-1/2,x,z)k_{\mu}(1/2,y,z)d\mu(z),\qquad\,\,t>1/2,\;\,x,y\in\R^N.
\end{eqnarray*}
By \eqref{estim-(0,1]}, the function $k_{\mu}(1/2,y,\cdot)$ belongs to $L^2_{\mu}$. Hence,
\begin{eqnarray*}
k_{\mu}(t,x,y)=(e^{(t-\frac{1}{2})A_{\mu}}k_{\mu}(1/2,y,\cdot))(x),\qquad\;\,t>1/2,\;\,x,y\in\R^N.
\end{eqnarray*}
Using again the semigroup law and the symmetry we deduce that
\begin{align*}
k_{\mu}(t,x,x)=&\int_{\R^N}|k_{\mu}(t/2,x,y)|^2d\mu(y)\\
\le & Me^{\lambda_0(t-1)}\|k_{\mu}(1/2,x,\cdot)\|_{L^2_{\mu}}^2\\
= & Me^{\lambda_0(t-1)}k_{\mu}(1,x,x).
\end{align*}
So, by applying \eqref{estim-(0,1]} to $k_{\mu}(1,x,x)$ and using the inequality $$k_{\mu}(t,x,y)\le (k_{\mu}(t,x,x))^{1/2}(k_{\mu}(t,y,y))^{1/2},$$ one obtains \eqref{tge1}.
%
\end{proof}

\begin{remark}
It follows from Proposition \ref{on-diagonal} that the estimates obtained for the heat kernel $k$ in Theorem \ref{th:stima-kernel-via-autofunz} could be sharp in the space variables but certainly not in the time variable as we will prove in Proposition \ref{sharp-t}.
\end{remark}
\begin{remark}
In the above proof we use the Sobolev inequality
$$\|u\|^2_{L^{2^*}_\mu}\le C\|\nabla u\|_2^2$$
which holds in $D(a_\mu)$ but not in $H$ (consider for example the case where $\alpha>\beta+N$ and $u=1$).
\end{remark}

As a consequence of Theorem \ref{th:stima-kernel-via-autofunz} we deduce some estimates for the eigenfunctions.
\begin{corollary}
If the assumptions of Theorem \ref{th:stima-kernel-via-autofunz} hold, then all normalized eigenfunctions $\psi_j$ of $A_2$ satisfy
\begin{eqnarray*}
|\psi_j(x)|\le C_j|x|^{\frac{\alpha-\beta}{4}-\frac{N-1}{2}}
e^{-\frac{2}{\beta-\alpha+2}|x|^{\frac{\beta-\alpha+2}{2}}},
\end{eqnarray*}
for all $x\in\R^N\setminus B_1$, $j\in\N$ and a constant $C_j>0$.
\end{corollary}
\begin{proof}
Let $\lambda_j$ be an eigenvalue of $A_2$ and denote by $\psi_j$ any normalized (i.e. $\|\psi_j\|_{L^2(\R^N)}$ $=1$) eigenfunction associated to $\lambda_j$. Then, as in the proof of Theorem \ref{th:stima-kernel-via-autofunz}, we have
\begin{align*}
e^{\lambda_j t}|\psi_j(x)|=&\left|\int_{\R^N}k_\mu(t,x,y)\psi_j(y)\,d\mu(y)\right|\\
\le & \left(\int_{\R^N}k_\mu(t,x,y)^2d\mu(y)\right)^{\frac{1}{2}}\|\psi_j\|_{L^2_\mu}\\
= & (k_\mu(2t,x,x))^{\frac{1}{2}},
\end{align*}
for $t>0$ and $x\in\R^N$. So, the estimates follow from Theorem \ref{th:stima-kernel-via-autofunz}.
\end{proof}

\begin{remark}
It is possible to obtain better estimates of the kernels $k$ with respect to the time variable $t$ for small $t$. In fact if we denote by $S(\cdot)$ the semigroup generated by $(1+|x|^\alpha)\Delta$ in $C_b(\R^N)$, which is given by a kernel $p$, then by domination we have $0<k(t,x,y)\le p(t,x,y)$ for $t>0$ and $x,\,y\in \R^N$. So, by \cite[Theorem 2.6 and Theorem 2.14]{met-spi3}, it follows that
\begin{eqnarray}\label{eq-improved}
k(t,x,y) &\le & Ct^{-N/2}(1+|x|)^{2-N}(1+|y|)^{2-N-\alpha},\quad \alpha>4, \\
k(t,x,y) &\le & Ct^{-N/2}(1+|x|^\alpha)^{\frac{2-N}{4}}(1+|y|^\alpha)^{\frac{2-N}{4}-1},\quad 2<\alpha\le 4 \nonumber
\end{eqnarray}
for $0<t\le 1,\,x,\,y\in \R^N$.
\end{remark}

Using a domination argument and \cite[Proposition 2.10]{met-spi3} we can improve the estimate \eqref{eq-improved}.

\begin{proposition}\label{sharp-t}
If $\alpha \ge 2,\,\beta>\alpha -2$ and $N>2$, then the kernel $k_\mu$ satisfies
$$k_\mu(t,x,y)\le Ct^{-N/2}(1+|x|^\alpha)^{\frac{2-N}{4}}(1+|y|^\alpha)^{\frac{2-N}{4}}$$
for $0<t\le 1$ and $x,\,y\in \R^N$.
\end{proposition}
\begin{proof}
It suffices to consider the case $\alpha >4$.\\
By domination one sees that weighted Nash inequalities given in \cite[Proposition 2.10]{met-spi3} hold for the quadratic form $a_\mu$. Hence, by \cite[Corollary 2.8]{bakry}, the results is proved provided that the function $\varphi(x)=(1+|x|^\alpha)^{\frac{2-N}{4}}$ is a Lyapunov function in the sense of \cite[Definition 2.1]{met-spi3}.\\
A simple computation yields
\begin{align*}
&A\varphi =\left[ \gamma(N+\alpha-2)|x|^{\alpha-2}+\gamma(\gamma-\alpha)|x|^{\alpha-2}\frac{|x|^\alpha}{1+|x|^\alpha} \right]\varphi (x)\\
&\quad =\left[ \gamma \left( \gamma+N-2 \right) \frac{|x|^{2\alpha-2}}{1+|x|^\alpha}
 +\gamma(\alpha-2+N) \frac{|x|^{\alpha-2}}{1+|x|^\alpha}
 -|x|^\beta \right]\varphi(x)\\
 &\quad \leq \left[
 \gamma \left( \gamma+N-2 \right) \frac{|x|^{2\alpha-2}}{1+|x|^\alpha}
 -|x|^\beta \right]\varphi(x),
\end{align*}
where $\gamma :=\frac{\alpha(2-N)}{4}$. We note that $\gamma <2-N$, since $\alpha >4$. Now, using the fact that $\beta >\alpha -2$, we deduce that
\begin{align*}
& \gamma \left( \gamma+N-2 \right) \frac{|x|^{2\alpha-2}}{1+|x|^\alpha} \leq
\gamma \left( \gamma+N-2 \right)|x|^{\alpha-2}
 \leq |x|^\beta +\kappa
\end{align*}
for some $\kappa>0$. Thus, $A\varphi \le \kappa \varphi$. Using the same arguments as in
\cite[Lemma 2.13]{met-spi3} we obtain that $\varphi$ is Lyapunov function for $A$.
\end{proof}

As in \cite{lor-rhan} heat kernel estimates can be also obtained for a more general class of elliptic operators.

Let us consider the operator $B$, defined on smooth functions $u$ by
\begin{eqnarray*}
Bu=(1+|x|^\alpha)\sum_{j,k=1}^ND_k(a_{kj}D_ju)-Wu,
\end{eqnarray*}
under the following set of assumptions:

\begin{hyp}\label{hyp-1}

\noindent
\begin{enumerate}
\item
the coefficients $a_{kj}=a_{jk}$ belong to $C_b(\R^N)\cap W^{1,\infty}_{\rm loc}(\R^N)$ for any $j,k=1,\ldots,N$
and there exists a positive constant $\eta$ such that
\begin{eqnarray*}
\eta |\xi|^2  \le \sum_{j,k=1}^Na_{kj}(x)\xi_k\xi_j,\qquad\;\,x,\xi \in \R^N;
\end{eqnarray*}
\item
$W\in L^1_{\rm loc}(\R^N)$ satisfies $W(x)\ge |x|^\beta$  for any $x\in \R^N$ and some $\beta>\alpha -2$;
\item
$\alpha\ge 2$ and $D_j a_{kj}(x) = {\rm o}(|x|^{\frac{\beta-\alpha}{2}})$ as $|x| \to \infty$.
\end{enumerate}
\end{hyp}

On $L^2_\mu$ we define the bilinear form
\begin{align*}
b_\mu(u,v)=\sum_{j,k=1}^N\int_{\R^N}a_{kj}D_kuD_j\overline{v}\,dx+\int_{\R^N}Wu\overline{v}\,d\mu,\qquad\;\,u,v\in D(b_{\mu}),
\end{align*}
where $D(b_{\mu})=\overline{C_c^\infty(\R^N)}^{\|\cdot \|_{\mathcal{H}}}$ with $\mathcal{H}$ the Hilbert space
$$\mathcal{H}=\{u\in L^2_{\mu}: W^{1/2}u\in L^2_{\mu},\,\nabla u\in (L^2(\R^N))^N\}.$$
Since $b_{\mu}$ is a symmetric, accretive and closable form, we can associate a positive strongly continuous semigroup $S_\mu(\cdot)$ in $L^2_{\mu}$.
The same arguments as in the beginning of this section show
that the infinitesimal generator $B_{\mu}$ of this semigroup is the realization in $L^2_{\mu}$
of the operator $B$ with domain $D(B_{\mu})=\{u\in D(b_{\mu})\cap W^{2,2}_{\rm loc}(\R^N): Bu\in L^2_{\mu}\}$.
Let us denotes by $p_\mu$ the heat kernel associated to $S_\mu(\cdot)$.

We will also need the bilinear form
\begin{eqnarray*}
a_{\mu,\theta}(u,v)=\int_{\R^N}\nabla u\cdot\nabla\overline{v}\,dx+\theta^2\int_{\R^N}Vu\overline{v}d\mu,\qquad\;\,u,v\in D(a_{\mu,\theta})=
D(a_{\mu}).
\end{eqnarray*}
The same arguments as in the proof of Theorem \ref{th:stima-kernel-via-autofunz} can be used to show that the kernel $k_{\mu,\theta}$ of the analytic semigroup associated to the form $a_{\mu,\theta}$ in $L^2_{\mu}$ satisfies
\begin{equation}
0<k_{\mu,\theta}(t,x,y)\le  K_{\theta}e^{\lambda_{0,\theta}t}e^{\tilde c_{\theta}t^{-b}}\psi_{\theta}(x)\psi_{\theta}(y),\qquad t>0,\,x,y\in\R^N,
\label{kernel-estim-theta}
\end{equation}
where $\tilde c_{\theta}$ and $K_{\theta}$ are positive constants, $\lambda_{0,\theta}$
is the largest (negative) eigenvalue of the minimal realization of
operator $A_{\theta}:=(1+|x|^\alpha)\Delta -\theta|x|^\beta$ in $L^2(\R^N)$, and $\psi_{\theta}$ is a corresponding
positive and bounded eigenfunction. Moreover, there exist $C_{1,\theta},\,C_{2,\theta}>0$ such that
\begin{eqnarray*}
C_{1,\theta}\le |x|^{-\frac{\alpha-\beta}{4}+\frac{N-1}{2}}
e^{-\theta\frac{2}{\beta-\alpha+2}|x|^{\frac{\beta-\alpha+2}{2}}}
\psi_{\theta}(x)\le C_{2,\theta},
\end{eqnarray*}
for any $x\in\R^N\setminus B_1$.

Using Theorem \ref{th:stima-kernel-via-autofunz} and arguing as in \cite{ouha-rhan} and \cite[Theorem 3.9]{lor-rhan} we obtain the following heat kernel estimate.
\begin{theorem}
\label{thm-ultimo}
Assume that Hypotheses $\ref{hyp-1}$ are satisfied and
let
\begin{eqnarray*}
\Lambda:=\sup_{x,\xi\in\R^N\setminus\{0\}}|\xi|^{-2}\sum_{j,k=1}^Na_{kj}(x)\xi_k\xi_j.
\end{eqnarray*}
Then, for any $\theta\in (0,\Lambda^{-1/2})$, we have
\begin{align*}
p_\mu(t,x,y)\le &M_{\theta}e^{\lambda_{0,\theta}t}e^{c_{\theta}t^{-b}}(|x||y|)^{\frac{\alpha-\beta}{4}-\frac{N-1}{2}}\notag\\
&\qquad\times
e^{-\theta\frac{2}{\beta-\alpha+2}|x|^{\frac{\beta-\alpha+2}{2}}}
e^{-\theta\frac{2}{\beta-\alpha+2}|y|^{\frac{\beta-\alpha+2}{2}}}
\end{align*}
for any $t>0$ and $x,y\in\R^N\setminus B_1$, where $M_{\theta}$, $c_{\theta}$ are positive constants, $b=\frac{\beta-\alpha+2}{\beta+\alpha-2}$, and
$\lambda_{0,\theta}$ is the largest eigenvalue of the operator $(1+|x|^\alpha)\Delta -\theta|x|^\beta$.
\end{theorem}
\begin{proof}
For the reader's convenience, we give the main ideas of the proof.

Proving the above estimate is equivalent to showing that
\begin{equation}\label{eq-new}
\phi_\theta(x)^{-1}p_\mu(t,x,y)\phi_\theta(y)^{-1}\le M_\theta e^{\lambda_{0,\theta}t}e^{c_{\theta}t^{-b}},\quad t>0,\,x,y\in \R^N,
\end{equation}
where $\phi_\theta$ is any smooth function satisfying
$$\phi_\theta(x)=|x|^{\frac{\alpha-\beta}{4}-\frac{N-1}{2}}
e^{-\theta\frac{2}{\beta-\alpha+2}|x|^{\frac{\beta-\alpha+2}{2}}},\quad x\in \R^N\setminus B_1.$$
If we denote by $T_{\phi_\theta}:L^2_{\phi_\theta^2\mu}\to L^2_\mu$ the isometry defined by $T_{\phi_\theta}f=\phi_\theta f$, then the left hand side of \eqref{eq-new} is the kernel of the semigroup $(T_{\phi_\theta}^{-1}e^{tB_{\mu}}T_{\phi_\theta})_{t\ge 0}$ in $L^2_{\phi_\theta^2\mu}$.
It is clear that this semigroup is associated with the form $\tilde{b}_\mu(u,v)=b_\mu(\phi_\theta u,\phi_\theta v)$ for $u,\,v\in D(\tilde{b}_\mu):=\{u\in L^2_{\phi_\theta^2\mu}:\phi_\theta u\in D(b_\mu)\}$.

As in the proof of Theorem \ref{th:stima-kernel-via-autofunz}, it suffices to establish \eqref{eq-new} for $t\in (0,1]$. To this purpose one has to prove, as in the proof of \cite[Theorem 3.9]{lor-rhan}, the following assertions:
\begin{enumerate}
\item[(i)]
$\min\{u,1\}\in D(\tilde{a}_{\mu,\theta})$ (resp. $D(\tilde{b}_{\mu})$) for any nonnegative $u\in D(\tilde{a}_{\mu,\theta})$ (resp.
$D(\tilde{b}_{\mu})$);
\item[(ii)]
the semigroup $(T_{\phi_\theta}^{-1}e^{tB_{\mu}}T_{\phi_\theta})_{t\ge 0}$ and the semigroup
$(T_{\phi_\theta}^{-1}e^{tA_{\mu ,\theta}}T_{\phi_\theta})_{t\ge 0}$, associated to the form $\tilde{a}_{\mu,\theta}=a_{\mu,\theta}(\phi_\theta\cdot,\phi_\theta\cdot)$
with domain $D(\tilde{a}_{\mu,\theta})=D(\tilde{b}_{\mu})$, are positive, they map $L^{\infty}(\R^N)$ into itself and satisfy the estimates
\begin{eqnarray*}
\|T_{\phi_\theta}^{-1}e^{tB_{\mu}}T_{\phi_\theta}\|_{L(L^{\infty}(\R^N))}\le e^{C_1t},\;\,
\|T_{\phi_\theta}^{-1}e^{tA_{\mu ,\theta}}T_{\phi_\theta}\|_{L(L^{\infty}(\R^N))}\le e^{C_1t},\;\,t>0,
\end{eqnarray*}
for some positive constant $C_1$;
\item[(iii)]
the Log-Sobolev inequality
\begin{align}
\;\;\int_{\R^N}u^2(\log u)\phi_\theta^2d\mu\le\varepsilon\tilde{b}_{\mu}(u,u)
+\|u\|_{L^2_{\phi_\theta^2\mu}}^2\log\|u\|_{L^2_{\phi_\theta^2\mu}}+c_\theta(1+\varepsilon^{-b})\|u\|_{L^2_{\phi_\theta^2\mu}}^2
\label{LSI}
\end{align}
holds true for any nonnegative $u\in D(\tilde{b}_{\mu})\cap L^1_{\phi_\theta^2\mu}\cap L^{\infty}(\R^N)$, where $c_\theta$ is the constant in
\eqref{eq-new}.
\end{enumerate}
So, applying \eqref{LSI} and combining \cite[Lemma 2.1.2, Cor. 2.2.8 and Ex. 2.3.4]{davies},
estimate \eqref{eq-new} follows with $t\in (0,1]$.

The proof of (i), (ii) and (iii) is similar to the one in \cite[Theorem 3.9]{lor-rhan}. The proof of (iii) is based on the estimate
$\tilde{b}_{\mu}(u,u)\ge \min\{\mu,\theta^{-1}\}\tilde{a}_{\mu,\theta}(u,u)$ which holds for any $u\in D(\tilde{b}_{\mu})
\subset D(\tilde{a}_{\mu,\theta})$, and \eqref{kernel-estim-theta}.
\end{proof}

\bibliography{bibfile}{}
\bibliographystyle{amsplain}

\end{document}